\documentclass[11pt]{amsart}
\usepackage{a4wide}
\usepackage{hyperref,amsmath,amsfonts,mathscinet,amssymb,amsthm}

\newtheorem{theorem}{Theorem}[section]

\numberwithin{equation}{section}

\expandafter\let\expandafter\oldproof\csname\string\proof\endcsname
\let\oldendproof\endproof
\renewenvironment{proof}[1][\proofname]{%
  \oldproof[\bf #1]%
}{\oldendproof}

\def\Z{\mathbb Z}

\def\esssup{\operatornamewithlimits{ess\,\sup}}

\begin{document}

\title{The two-weight Hardy inequality: a new elementary and universal proof}

\author{Amiran Gogatishvili and Lubo\v s Pick}

\email[A.~Gogatishvili]{gogatish@math.cas.cz}
\urladdr{0000-0003-3459-0355}
\email[L.~Pick]{pick@karlin.mff.cuni.cz}
\urladdr{0000-0002-3584-1454}

\address{Amiran Gogatishvili,
 Institute of Mathematics of the
 Czech Academy of Sciences,
 \v Zitn\'a~25,
 115~67 Praha~1,
 Czech Republic \newline
 L. N. Gumilyov Eurasian National University,
5 Munaytpasov St., 010008 Nur-Sultan, Kazakhstan}

\address{Lubo\v s Pick,
 Department of Mathematical Analysis,
	Faculty of Mathematics and Physics,
	Charles University,
	Sokolovsk\'a~83,
	186~75 Praha~8,
	Czech Republic}
	

\thanks{This research was supported by the grant P201-18-00580S of the Czech Science Foundation. The research of  A. Gogatishvili was also supported by Czech Academy of Sciences RVO: 67985840}

\begin{abstract}
We give a~new proof of the known criteria for the inequality
    \begin{equation*}
        \left(\int_{0}^{\infty}\left(\int_{0}^{t}f\right)^{q}w(t)\,dt\right)^{\frac{1}{q}}
        \leq C
        \left(\int_{0}^{\infty}f^{p}v\right)^{\frac{1}{p}}.
\end{equation*}
The innovation is in the elementary nature of the proof and its versatility.
\end{abstract}

\date{\today}

\maketitle

\section{Introduction}

Consider the two-weight Hardy inequality
\begin{equation}\label{E:1}
    \left(\int_{0}^{\infty}\left(\int_{0}^{t}f\right)^{q}w(t)\,dt\right)^{\frac{1}{q}}
        \leq C
        \left(\int_{0}^{\infty}f^{p}v\right)^{\frac{1}{p}},
\end{equation}
in which $C$ is a~positive constant independent of a nonnegative measurable function $f$ on $(0,\infty)$, $v$ and $w$ are fixed nonnegative measurable functions on $(0,\infty)$ (weights), $p\in[1,\infty)$, and $q\in(0,\infty)$. The requirement $p\in[1,\infty)$ is reasonable since for $p\in(0,1)$ there are functions in weighted $L^p$ which are not locally integrable.

The problem of characterizing pairs of weights for which~\eqref{E:1} is true has a long and rich history and it would be impossible to mention here every contribution. For $p=q>1$, $v=1$,
$w(t)=t^{-q}$ and $C=p'$, it is just the boundedness of the integral averaging operator on $L^{p}(0,\infty)$, a result almost one century old, which appears in classical Hardy's papers in 1920's, see~\cite{HLP:88}. The beginning of investigation of a~general weighted case goes back to 1950's, and it starts with the paper by Kac and Krein~\cite{KacKre:58}  in which a characterization for $p=q=2$ and $v=1$ can be found. In 1950's and 1960's,
plenty of partial results were obtained by Beesack, see e.g.~\cite{Bee:61}. In late 1960's and in 1970's, a boom in the so-called \textit{convex case} ($p\le q$, named after the convexity of $t\mapsto t^{\frac{q}{p}}$) was seen. For $p=q$, a characterization was obtained by Tomaselli~\cite{Tom:69}, Talenti~\cite{Tal:69} and Muckenhoupt~\cite{Muc:72}. It was extended to $p\leq q$ by Bradley~\cite{Bra:78}, the same result is also stated without proof
in~\cite{Kok:79}. Many authors referred further to an untitled and unpublished manuscript by Artola, and in~\cite{Rie:74}, a paper by D.W. Boyd and J.A. Erd\H{o}s was quoted, which most likely was never published. In any case,~\eqref{E:1} holds if and only if
\begin{equation*}
    \sup_{t\in(0,\infty)}\left(\int_{t}^{\infty}w\right)^{\frac{1}{q}}
        \left(\int_{0}^{t}v^{1-p'}\right)^{\frac{1}{p'}}<\infty \quad\text{for $1<p\le q$}
\end{equation*}
and
\begin{equation*}
    \sup_{t\in(0,\infty)}\left(\int_{t}^{\infty}w\right)^{\frac{1}{q}}
        \esssup_{s\in(0,t)}\frac{1}{v(s)}<\infty \quad\text{for $1=p\le q$.}
\end{equation*}
Here and throughout, if $p\in(0,\infty]$, then $p'$ denotes the conjugate exponent defined by $\frac{1}{p}+\frac{1}{p'}=1$. Observe that $1$ and $\infty$ are conjugate exponents and that $p'$ is negative when $p\in(0,1)$.

The non-convex case ($p>q$) turned out to be more difficult to handle, and it had to wait till 1980's and 1990's for appropriate treatment. The first characterization, for $1\le q<p<\infty$, was obtained by Maz'ya and
Rozin, see~\cite{Maz:11}, who proved that a necessary and sufficient condition is
\begin{equation*}
    \int_{0}^{\infty}\left(\int_{t}^{\infty}w\right)^{\frac{r}{q}}\left(\int_{0}^{t}v^{1-p'}\right)^{\frac{r}{q'}}v(t)^{1-p'}\,dt<\infty,
\end{equation*}
where $r=\frac{pq}{p-q}$. A universal characterization, sheltering both the convex and the non-convex cases and involving more general norms was obtained by Sawyer~\cite{Saw:84}, but the condition in the non-convex case is expressed in terms of a
discretized condition. While discretization techniques proved later to be of colossal theoretical importance, conditions expressed in terms of discretizing sequences are difficult to verify. Later,
Sinnamon~\cite{Sin:91} characterized the inequality for $0<q<1<p<\infty$. The criterion turns out to be the same as that of Maz'ya and Rozin but the proof, based on Halperin's level function, is very
different. The case $0<q<p=1$ was treated by Sinnamon and Stepanov~\cite{SiSt:96}, who moreover observed that, unless $p=1$, Sinnamon's and Mazya-Rozin's
results can be proved in a unified manner. The case $p=1$, however, still required separate treatment. In~\cite{BeGr:06}, restriction of~\eqref{E:1} to the cone of non-increasing functions is studied, together with its discrete version. Some ideas developed there are useful also for the unrestricted case.

In this note we present a short, uniform and elementary proof, in which
\begin{itemize}
\item all cases are covered,
\item $p>1$ is not separated from $p=1$,
\item only Fubini's theorem, H\"older's inequality, Minkowski's integral inequality and Hardy's lemma are used.
\end{itemize}

\section{The theorem and its proof}

\begin{theorem}
Let $v,w$ be weights on $(0,\infty)$, $p\in[1,\infty)$ and $q\in(0,\infty)$. For $t\in(0,\infty)$, denote
\begin{equation*}
    V(t) =
        \begin{cases}
            \left(\int_{0}^{t}v^{1-p'}\right)^{\frac{1}{p'}} &\text{if $p\in(1,\infty)$,}
                \\
            \esssup_{s\in(0,t)}\frac{1}{v(s)} &\text{if $p=1$,}
        \end{cases}
\end{equation*}
and
\begin{equation*}
    W(t)=\int_{t}^{\infty}w.
\end{equation*}
Then there exists a positive constant $C$ such that~\eqref{E:1} holds for every nonnegative measurable function $f$ on $(0,\infty)$ if and only if $A<\infty$, where
\begin{equation*}
    A = \begin{cases}
        \sup\limits_{t\in(0,\infty)} V(t)W(t)^{\frac{1}{q}}
            & \text{if $p\le q$,}
                \\
        \int_{0}^{\infty}W^{\frac{p}{p-q}}\,dV^{\frac{pq}{p-q}}
            & \text{if $p>q$,}
        \end{cases}
\end{equation*}
in which the latter integral should be understood in the Lebesgue--Stieltjes sense with respect to the (monotone) function $V^{\frac{pq}{p-q}}$.
\end{theorem}

\begin{proof}
\textit {Sufficiency.} Fix $\varepsilon\in(0, 1)$. We claim that, for every nonnegative measurable function $f$ on $(0,\infty)$, one has
\begin{equation}\label{E:3}
    \int_{0}^{t}f \lesssim
    \left(\int_{0}^{t}f^{p}V^{\varepsilon p}v\right)^{\frac{1}{p}}V(t)^{1-\varepsilon} \quad\text{for $t>0$.}
\end{equation}
(We write $\lesssim$ when the expression to the left of it is majorized by a constant times that on the right.) To show~\eqref{E:3}, fix $t\in(0,\infty)$. If $p\in(1,\infty)$, then, by H\"older's inequality,
\begin{align*}
    \int_{0}^{t}f & = \int_{0}^{t}fV^{\varepsilon}v^{\frac{1}{p}}V^{-\varepsilon}v^{-\frac{1}{p}}
        \le \left(\int_{0}^{t}f^{p}V^{\varepsilon p}v\right)^{\frac{1}{p}}
        \left(\int_{0}^{t}V^{-\varepsilon p'}v^{1-p'}\right)^{\frac{1}{p'}}.
\end{align*}
By a change of variables, we obtain
\begin{equation*}
    \int_{0}^{t}V^{-\varepsilon p'}v^{1-p'} = \int_{0}^{t} \left(\int_{0}^{s}v^{1-p'}\right)^{-\varepsilon}v^{1-p'}\,ds = \frac{1}{1-\varepsilon}\left(\int_{0}^{s}v^{1-p'}\right)^{1-\varepsilon}
    = \frac{1}{1-\varepsilon} V(t)^{(1-\varepsilon)p'},
\end{equation*}
hence
\begin{equation*}
    \int_{0}^{t}f \lesssim \left(\int_{0}^{t}f^{p}V^{\varepsilon p}v\right)^{\frac{1}{p}}V(t)^{1-\varepsilon},
\end{equation*}
and~\eqref{E:3} follows. If $p=1$, then we get~\eqref{E:3} from
\begin{equation*}
    \int_{0}^{t}f = \int_{0}^{t}fv^{-\varepsilon}vv^{-1+\varepsilon}
        \le \left(\int_{0}^{t}fV^{\varepsilon}v\right) V(t)^{1-\varepsilon}.
\end{equation*}

Let $p\le q$. Then $A<\infty$ implies $V\lesssim W^{-\frac{1}{q}}$. Using this and~\eqref{E:3}, we get
\begin{equation*}
    \int_{0}^{t}f \lesssim \left(\int_{0}^{t}f^{p}W^{-\frac{\varepsilon p}{q}}v\right)^{\frac{1}{p}}W(t)^{\frac{\varepsilon-1}{q}}\quad\text{for $t>0$.}
\end{equation*}
Raising to $q$ and integrating with respect to $w(t)\,dt$, we obtain
\begin{equation*}
    \int_{0}^{\infty} \left(\int_{0}^{t}f\right)^{q}w(t)\,dt
         \lesssim \int_{0}^{\infty} \left(\int_{0}^{t}f(s)^{p}W(s)^{-\frac{\varepsilon p}{q}}v(s)\,ds\right)^{\frac{q}{p}}W(t)^{\varepsilon-1}w(t)\,dt.
\end{equation*}
Next we apply Minkowski's integral inequality (note that $\frac{q}{p}\ge 1$ and all the expressions in the play are nonnegative) in the form
\begin{equation*}
    \int_{0}^{\infty}\left(\int_{0}^{\infty} F(s,t)\,d\mu_1(s)\right)^{\frac{q}{p}}\,d\mu_2(t)
        \le \left(\int_{0}^{\infty}\left(\int_{0}^{\infty}F(s,t)^{\frac{q}{p}}\,d\mu_2(t)\right)^{\frac{p}{q}}\,d\mu_1(s)\right)^{\frac{q}{p}},
\end{equation*}
in which $F(s,t)=\chi_{(0,t)}(s)f(s)^{p}$, $\chi$ denotes the characteristic function, $d\mu_1(s)=W(s)^{-\frac{\varepsilon p}{q}}v(s)ds$ and $d\mu_2(t)=W(t)^{\varepsilon-1}w(t)\,dt$. We thus obtain
\begin{align*}
        &\int_{0}^{\infty} \left(\int_{0}^{t}f(s)^{p}W(s)^{-\frac{\varepsilon p}{q}}v(s)\,ds\right)^{\frac{q}{p}}W(t)^{\varepsilon-1}w(t)\,dt
            \\
        & \le \left(\int_{0}^{\infty}f(s)^{p}W(s)^{-\frac{\varepsilon p}{q}}v(s)\left(\int_{s}^{\infty}W(t)^{\varepsilon-1}w(t)\,dt\right)^{\frac{p}{q}}\,ds\right)^{\frac{q}{p}}
            \\
        &\approx \left(\int_{0}^{\infty}f^{p}v\right)^{\frac{q}{p}}.
\end{align*}
 (We write $\approx$ when both $\lesssim$ and $\gtrsim$ apply.) Altogether, we arrive at
\begin{equation*}
    \int_{0}^{\infty} \left(\int_{0}^{t}f\right)^{q}w(t)\,dt \lesssim \left(\int_{0}^{\infty}f^{p}v\right)^{\frac{q}{p}},
\end{equation*}
and~\eqref{E:1} follows.

Let $p>q$. Fix $\alpha\in(0,\infty)$. We shall use the symbol $V(\infty)$ for $\lim_{t\to\infty} V(t)$ (this limit always exists, either finite or infinite, owing to the monotonicity of $V$). By~\eqref{E:3},
\begin{align*}
    \int_{0}^{\infty} & \left(\int_{0}^{t}f\right)^{q}w(t)\,dt
        \lesssim \int_{0}^{\infty}\left(\int_{0}^{t}f^{p}V^{\varepsilon p}v\right)^{\frac{q}{p}}V(t)^{-\alpha q }V(t)^{(1-\varepsilon+\alpha)q }w(t)\,dt
            \\
        & \lesssim \int_{0}^{\infty}\left(\int_{0}^{t}f^{p}V^{\varepsilon p}v\right)^{\frac{q}{p}}
          \left(V(t)^{-\alpha p}-V(\infty)^{-\alpha p}\right)^{\frac{q}{p}}V(t)^{(1-\varepsilon+\alpha)q}w(t)\,dt
            \\
        & + \int_{0}^{\infty}\left(\int_{0}^{t}f^{p}V^{\varepsilon p}v\right)^{\frac{q}{p}}
          V(t)^{(1-\varepsilon+\alpha)q }w(t)\,dt\cdot V(\infty)^{-\alpha q}
            = I + II.
\end{align*}
If $V(\infty)=\infty$, one has $II=0$. Since
\begin{equation*}
    V(t)^{(1-\varepsilon+\alpha)q}\approx \int_{0}^{t}V^{(1-\varepsilon+\alpha)q-\frac{pq}{p-q}}d(V^{\frac{pq}{p-q}})\quad\text{for $t>0$}
\end{equation*}
and
\begin{equation*}
  V(t)^{-\alpha p}-V(\infty)^{-\alpha p} = \int_{t}^{\infty} d(-V^{-\alpha p})\quad\text{for $t>0$,}
\end{equation*}
monotonicity and Fubini's theorem yield
\begin{align*}
    I & \lesssim \int_{0}^{\infty}\left(\int_{t}^{\infty}\left(\int_{0}^{s}f^{p}V^{\varepsilon p}v\right) d(-V^{-\alpha p})(s)\right)^{\frac{q}{p}}
           \left(\int_{0}^{t}V^{(1-\varepsilon+\alpha)q-\frac{pq}{p-q}}\,dV^{\frac{pq}{p-q}}\right)w(t)\,dt
            \\
         & \lesssim
            \int_{0}^{\infty}\left(\int_{0}^{t}\left(\int_{s}^{\infty}\left(\int_{0}^{\tau}f^{p}V^{\varepsilon p}v\right)d(-V^{-\alpha p})(\tau)\right)^{\frac{q}{p}}
            V(s)^{(1-\varepsilon+\alpha)q-\frac{pq}{p-q}}dV^{\frac{pq}{p-q}}(s)\right)\,w(t)\,dt
            \\
        & =
            \int_{0}^{\infty}\left(\int_{s}^{\infty}\left(\int_{0}^{\tau}f^{p}V^{\varepsilon p}v\right)d(-V^{-\alpha p})(\tau)\right)^{\frac{q}{p}}
             V(s)^{(1-\varepsilon+\alpha)q-\frac{pq}{p-q}}W(s)dV^{\frac{pq}{p-q}}(s).
\end{align*}
Thus, owing to $A<\infty$, H\"older's inequality, and Fubini's theorem,
\begin{align*}
        I & \lesssim
            \left(\int_{0}^{\infty}W^{\frac{p}{p-q}}dV^{\frac{pq}{p-q}}\right)^{\frac{p-q}{p}}
            \left(\int_{0}^{\infty}\left(\int_{s}^{\infty}
            \left(\int_{0}^{\tau}f^{p}V^{\varepsilon p}v\right)d(-V^{-\alpha p})(\tau)\right)V(s)^{(1-\varepsilon+\alpha) p-\frac{p^2}{p-q}}dV^{\frac{pq}{p-q}}(s)\right)^{\frac{q}{p}}
            \\
        & \lesssim
            \left(\int_{0}^{\infty}
            \left(\int_{0}^{\tau}f^{p}V^{\varepsilon p}v\right)\left(\int_{0}^{\tau}V^{(1-\varepsilon+\alpha)p-\frac{p^2}{p-q}}dV^{\frac{pq}{p-q}}\right)d(-V^{-\alpha p})(\tau)\right)^{\frac{q}{p}}
            \\
        & \approx
            \left(\int_{0}^{\infty}
            \left(\int_{0}^{\tau}f^{p}V^{\varepsilon p}v\right)
             V(\tau)^{(\alpha-\varepsilon)p}d(-V^{-\alpha p})(\tau)\right)^{\frac{q}{p}}
          \\
        & =
            \left(\int_{0}^{\infty}
            f(y)^{p}V(y)^{\varepsilon p}v(y)
             \int_{y}^{\infty}V^{(\alpha-\varepsilon)p}d(-V^{-\alpha p})dy\right)^{\frac{q}{p}}
                \approx \left(\int_{0}^{\infty}f^{p}v\right)^{\frac{q}{p}}.
\end{align*}
If $V(\infty)<\infty$, we have
\begin{align*}
    II & \le \int_{0}^{\infty}\left(\int_{0}^{t}f^{p}v\right)^{\frac{q}{p}}
              V(t)^{(1+\alpha)q }w(t)dt\cdot V(\infty)^{-\alpha q}
        \le \left(\int_{0}^{\infty}f^{p}v\right)^{\frac{q}{p}} \left(\int_{0}^{\infty}
              V^{(1+\alpha)q }w\right) V(\infty)^{-\alpha q}.
\end{align*}
Owing to $A<\infty$, Fubini's theorem, and H\"older's inequality, we get
\begin{align*}
  \int_{0}^{\infty} & V^{(1+\alpha)q} w \approx \int_{0}^{\infty}\left(\int_0^t V^{\alpha q + q-p'}v^{1-p'}\right)  w(t)\,dt
      =\int_{0}^{\infty} V^{\alpha q + q-p'}v^{1-p'}W
        \\
    & \lesssim
    \left(\int_{0}^{\infty} V^{\alpha p-p'}v^{1-p'}\right) ^{\frac{q}{p}}
        \left(\int_{0}^{\infty}W(t)^{\frac{p}{p-q}}\,dV^{\frac{pq}{p-q}}\right)^{\frac{p-q}{p}}
       \lesssim V(\infty)^{\alpha q},
\end{align*}
establishing $II\lesssim \left(\int_{0}^{\infty}f^{p}v\right)^{\frac{q}{p}}$. This shows sufficiency.

\textit {Necessity.} Let $p\le q$ and assume that~\eqref{E:1} holds. Fix $t\in(0,\infty)$. Then
\begin{equation*}
    \int_{0}^{\infty}\left(\int_{0}^{s}f\right)^{q}w(s)\,ds
        \ge \int_{t}^{\infty}\left(\int_{0}^{s}f\right)^{q}w(s)\,ds
        \ge W(t)\left(\int_{0}^{t}f\right)^{q}.
\end{equation*}
Therefore,~\eqref{E:1} yields
\begin{equation}\label{E:lower-for-C}
    C\ge W(t)^{\frac{1}{q}}\sup_{f\ge0}\frac{\int_{0}^{t}f}{\left(\int_{0}^{\infty}f^{p}v\right)^{\frac{1}{p}}}.
\end{equation}
We claim that
\begin{equation}\label{E:evaluation-of-supremum}
    \sup_{f\ge0}\frac{\int_{0}^{t}f}{\left(\int_{0}^{\infty}f^{p}v\right)^{\frac{1}{p}}} = V(t).
\end{equation}
Indeed, if $p>1$, then we have, by H\"older's inequality,
\begin{equation*}
    \int_{0}^{t}f = \int_{0}^{t}f v^{\frac{1}{p}} v^{-\frac{1}{p}} \le \left(\int_{0}^{t}f^{p}v\right)^{\frac{1}{p}}\left(\int_{0}^{t}v^{-\frac{p'}{p}}\right)^{\frac{1}{p'}} \le \left(\int_{0}^{\infty}f^{p}v\right)^{\frac{1}{p}}V(t)
\end{equation*}
for every measurable $f\ge0$. On the other hand, this inequality is saturated by the choice $f=v^{1-p'}\chi_{(0,t)}$, since $f^pv=f$, and, consequently,
\begin{equation*}
    \int_{0}^{t}f = \left(\int_{0}^{t}f^{p}v\right)^{\frac{1}{p}}\left(\int_{0}^{t}f^{p}v\right)^{\frac{1}{p'}} = \left(\int_{0}^{\infty}f^{p}v\right)^{\frac{1}{p}}V(t).
\end{equation*}
If $p=1$, then we, once again, obtain
\begin{equation*}
    \int_{0}^{t}f = \int_{0}^{t}f v v^{-1} \le V(t) \int_{0}^{t}f v \le V(t) \int_{0}^{\infty}f v
\end{equation*}
for every measurable $f\ge0$. In order to saturate this inequality, fix any $\lambda< V(t)$. Then there exists a set $E\subset(0,t)$ of positive measure such that $\frac{1}{v}\ge\lambda$ on $E$. Set $f=\frac{\chi_{E}}{v}$. Then
\begin{equation*}
    \int_{0}^{t}f = \int_{E}\frac{1}{v} \ge \lambda|E| = \lambda\int_{0}^{\infty}fv.
\end{equation*}
On letting $\lambda\to V(t)_{-}$, we get
\begin{equation*}
    \int_{0}^{t}f \ge V(t)\int_{0}^{\infty}fv.
\end{equation*}
In any case,~\eqref{E:evaluation-of-supremum} follows. Since $t$ was arbitrary, plugging~\eqref{E:evaluation-of-supremum} into~\eqref{E:lower-for-C} yields
\begin{equation*}
    C\ge\sup_{t\in(0,\infty)}W(t)^{\frac{1}{q}}\sup_{f\ge0}\frac{\int_{0}^{t}f}{\left(\int_{0}^{\infty}f^{p}v\right)^{\frac{1}{p}}} = \sup_{t\in(0,\infty)}W(t)^{\frac{1}{q}}V(t),
\end{equation*}
establishing $A<\infty$.

Let $p>q$ and $p>1$, denote $r=\frac{pq}{p-q}$ and $B = \int_{0}^{\infty}V^{r}W^{\frac{r}{p}} w$. Let $\theta\in(\frac{r}{p'},\infty)$ and set
\begin{equation*}
    f(t) = \left(\int_{t}^{\infty}W^{\frac{r}{p}}wV^{r-\theta p'}\right)^{\frac{1}{p}} V(t)^{(\theta-1)(p'-1)}v(t)^{1-p'}\quad\text{for $t>0$.}
\end{equation*}
By Fubini's theorem,
\begin{align*}
    & \int_{0}^{\infty}f^{p}v = \int_{0}^{\infty} \left(\int_{t}^{\infty}W^{\frac{r}{p}}wV^{r-\theta p'}\right)V(t)^{(\theta-1)p'}v(t)^{1-p'}dt
                \\
        & = \int_{0}^{\infty} W(s)^{\frac{r}{p}}w(s)V(s)^{r-\theta p'}\left(\int_{0}^{s}V^{(\theta-1)p'}v^{1-p'}\right)ds \approx B.
\end{align*}
On the other hand, by monotonicity,
\begin{align*}
    & \int_{0}^{\infty}\left(\int_{0}^{t}f\right)^{q}w(t)dt
     \ge \int_{0}^{\infty}\left(\int_{0}^{t}V^{(\theta-1)(p'-1)}v^{1-p'}\right)^{q}\left(\int_{t}^{\infty}W^{\frac{r}{p}}wV^{r-\theta p'}\right)^{\frac{q}{p}}w(t)dt
            \\
    & \ge    \int_{0}^{\infty}\left(\int_{0}^{t}V^{(\theta-1)(p'-1)+\frac{r}{p}-\frac{\theta p'}{p}}v^{1-p'}\right)^{q}\left(\int_{t}^{\infty}W^{\frac{r}{p}}w\right)^{\frac{q}{p}}w(t)dt
    \approx B.
 \end{align*}
Altogether,~\eqref{E:1} implies $B^{\frac{1}{q}} \lesssim B^{\frac{1}{p}}$. Using a standard approximation argument, we obtain $B^{\frac{1}{r}}<\infty$, hence $B<\infty$. Since $A\approx B$ owing to integration by parts, we get $A<\infty$.

Finally, let $p=1$ and $p>q$. Fix some $\sigma>1$ and define
\begin{equation*}
    E_k = \{t\in(0,\infty) : \sigma^k<V(t)\le \sigma^{k+1}\} \quad \text{for $k\in\Z$.}
\end{equation*}
Set $\mathbb A = \{k\in\Z : E_k\neq\emptyset\}$. Then $(0,\infty)=\bigcup_{k\in\mathbb A}E_k$, in which the union is disjoint and each $E_k$ is a nondegenerate interval (which could be either open or closed at each end) with endpoints $a_k$ and $b_k$, $a_k<b_k$. For every $k\in\mathbb A$, we find $\delta_k > 0$ so that $a_k+\delta_k<b_k$ and
\begin{equation}\label{E:small-epsilons}
    \int_{a_k}^{b_k}W^{\frac{q}{1-q}}w \le \sigma \int_{a_k+\delta_k}^{b_k}W^{\frac{q}{1-q}}w,
\end{equation}
which is clearly possible, and then we define the set
\begin{equation*}
    G_k = \left\{t\in(a_k,a_k+\delta_k) : \frac{1}{v(t)} >\sigma^{k}\right\}.
\end{equation*}
Since $V$ is non-decreasing and left-continuous, $|G_k|>0$ for every $k\in\mathbb A$. Set $h = \sum_{k\in\mathbb A} \frac{\chi_{G_k}}{|G_k|}$. Then, for every $k\in\mathbb A$, one has
\begin{equation}\label{E:estimate-of-h}
    \int_{0}^{a_k+\delta_k}hv^{-1}V^{\frac{q}{1-q}} \ge \int_{a_k}^{a_k+\delta_k}hv^{-1}V^{\frac{q}{1-q}} = \frac{1}{|G_k|}\int_{G_k}V^{\frac{q}{1-q}}v^{-1}\ge \sigma^{\frac{k}{1-q}}.
\end{equation}

Fix $t\in(0,\infty)$. Then there is a uniquely defined $k\in\mathbb A$ such that $t\in(a_k,b_k]$. Consequently,
\begin{align*}
    \int_{0}^{t}hV^{\frac{q}{1-q}}& \le \sum_{j\in\mathbb A,\ j\le k}\frac{1}{|G_j|}\int_{G_j}V^{\frac{q}{1-q}}
       \le \sum_{j=-\infty}^{k}\sigma^{\frac{q(j+1)}{1-q}} = \frac{\sigma^{\frac{q(k+2)}{1-q}}}{\sigma^{\frac{q}{1-q}}-1}.
\end{align*}
On the other hand,
\begin{align*}
    \int_{0}^{t}\,dV^{\frac{q}{1-q}} \ge \int_{0}^{a_k}\,dV^{\frac{q}{1-q}} = V(a_k)^{\frac{q}{1-q}} \ge \sigma^{\frac{qk}{1-q}}.
\end{align*}
The last two estimates yield
\begin{equation}\label{E:upper-estimate-of-integral}
    \int_{0}^{t}hV^{\frac{q}{1-q}}\lesssim \int_{0}^{t}\,dV^{\frac{q}{1-q}} \quad\text{for $t>0$.}
\end{equation}
Since $W^{\frac{1}{1-q}}$ is non-increasing, we can apply Hardy's lemma (whose version for Lebesgue integrals can be found in~\cite[Chapter~2, Proposition~3.6]{Ben:88} - note that the proof presented there works verbatim for Lebesgue--Stieltjes integrals) to~\eqref{E:upper-estimate-of-integral} and get
\begin{equation}\label{E:hardy-lemma-estimate}
    \int_{0}^{\infty}hV^{\frac{q}{1-q}}W^{\frac{1}{1-q}}\lesssim \int_{0}^{\infty}W^{\frac{1}{1-q}}dV^{\frac{q}{1-q}}.
\end{equation}
Finally, using subsequently integration by parts, decomposition of $(0,\infty)$ into $\bigcup_{k\in\mathbb A}E_k$, the definition of $E_k$, the fact that each $E_k$ is an interval with endpoints $a_k$, $b_k$, \eqref{E:small-epsilons}, \eqref{E:estimate-of-h}, monotonicity of functions given by integrals, \eqref{E:1} applied to $p=1$ and $f=hv^{-1}V^{\frac{q}{1-q}}W^{\frac{1}{1-q}}$, and~\eqref{E:hardy-lemma-estimate}, we get
\begin{align*}
    & \int_{0}^{\infty}  W^{\frac{1}{1-q}}dV^{\frac{q}{1-q}}
        \le 2 \int_{0}^{\infty} V^{\frac{q}{1-q}} W^{\frac{q}{1-q}} w
        = 2 \sum_{k\in\mathbb A} \int_{E_k} V^{\frac{q}{1-q}} W^{\frac{q}{1-q}} w
        \lesssim  \sum_{k\in\mathbb A} \sigma^{\frac{(k+1)q}{1-q}} \int_{E_k} W^{\frac{q}{1-q}} w
            \\
    &
        \lesssim  \sum_{k\in\mathbb A} \sigma^{\frac{(k+1)q}{1-q}} \int_{a_k+\delta_k}^{b_k} W^{\frac{q}{1-q}} w
        \lesssim \sum_{k\in\mathbb A} \left(\int_{0}^{a_k+\delta_k} hv^{-1}V^{\frac{q}{1-q}}\right)^{q} \int_{a_k+\delta_k}^{b_k}W^{\frac{q}{1-q}} w
            \\
    &   \lesssim \sum_{k\in\mathbb A} \int_{a_k+\delta_k}^{b_k}\left(\int_{0}^{t} hv^{-1}V^{\frac{q}{1-q}}\right)^{q} W(t)^{\frac{q}{1-q}} w(t) \, dt
        \lesssim \int_{0}^{\infty}\left(\int_{0}^{t} hv^{-1}V^{\frac{q}{1-q}}W^{\frac{1}{1-q}}\right)^{q}w(t)\, dt
            \\
    &
       \lesssim \left(\int_{0}^{\infty}hV^{\frac{q}{1-q}}W^{\frac{1}{1-q}}\right)^{q}
        \lesssim \left(\int_{0}^{\infty}W^{\frac{1}{1-q}}\,dV^{\frac{q}{1-q}}\right)^{q},
\end{align*}
in which the multiplicative constants depend only on $C$ and $q$. This establishes, via a standard approximation argument, that $A^{1-q}<\infty$, which in turn yields $A<\infty$. The proof is complete.
\end{proof}



\

\end{document}